\documentclass[a4paper,final]{article}

\usepackage[T1]{fontenc}
\usepackage[utf8]{inputenc}
\usepackage[francais]{layout}
\usepackage[francais,english]{babel}

\usepackage{amsmath,amssymb,amsfonts,mathabx,stmaryrd}
\usepackage{amsthm}
\usepackage{xargs}
\usepackage{mdwlist}
\usepackage[poly, all]{xy}
\usepackage{ifthen}

\usepackage{soul}
\setuldepth{a}

\usepackage{graphicx}
\usepackage{enumitem}

\usepackage{ltugcomn}

\usepackage{ifdraft}
\usepackage{comment}

\usepackage[final]{hyperref}

\makeatletter
\hypersetup{
    unicode=true,
    pdftoolbar=true,
    pdfmenubar=true,
    pdffitwindow=false,
    pdfstartview={FitH},
    pdftitle={\@title},
    pdfauthor={\@author},
    pdfsubject={},
    pdfcreator={},
    pdfproducer={},
    pdfkeywords={},
    pdfnewwindow=true,
    colorlinks,
    linkcolor=black,
    citecolor=black,
    filecolor=black,
    urlcolor=black
}
\makeatother

\newcommand{\thmlevel}{section}
\usepackage{aliascnt}
\def\NewTheorem#1{%
  \newaliascnt{#1}{thm}
  \newtheorem{#1}[#1]{\csname #1Name\endcsname}
  \aliascntresetthe{#1}
  \expandafter\def\csname #1autorefname\endcsname{\csname #1name\endcsname}
  \expandafter\def\csname #1Autorefname\endcsname{\csname #1Name\endcsname}
}

\newcommand{\theoremName}{\iflanguage{francais}{Th\'eor\`eme}{Theorem}}

\newcommand{\pbName}{\iflanguage{francais}{Probl\`eme}{Problem}}

\newcommand{\dfName}{\iflanguage{francais}{D\'efinition}{Definition}}

\newtheorem{thm}{\theoremName}[\thmlevel]

\newtheorem{thmintro}{\theoremName}

\newtheorem{dfintro}[thmintro]{\dfName}

\NewTheorem{lem}
\NewTheorem{prop}
\NewTheorem{cor}
\NewTheorem{conj}
\NewTheorem{pb}

\theoremstyle{definition}
\NewTheorem{df}
\NewTheorem{exo}

\theoremstyle{remark}
\NewTheorem{ex}
\NewTheorem{rmq}

\makeatletter
\renewenvironment{proof}[1][]{\par
  \pushQED{\qed}%
  \normalfont \topsep6\p@\@plus6\p@\relax
  \trivlist
  \item[\hskip\labelsep
        \bfseries
    \proofname\ifthenelse{\equal{#1}{}}{}{\textmd{ (#1)}}\@addpunct{.}]\ignorespaces
}{%
  \popQED\endtrivlist\@endpefalse
}
\makeatother

\makeatletter
\@addtoreset{thm}{section}
\makeatother
\let\originalleft\left
\let\originalright\right
\renewcommand{\left}{\mathopen{}\mathclose\bgroup\originalleft}
\renewcommand{\right}{\aftergroup\egroup\originalright}

\makeatletter
\def\DeclareMathBinOp{\@ifstar{\declaremathbinop@star}{\declaremathbinop@nostar}}
\def\declaremathbinop@star#1#2{\def#1{\test@subnexp@star#2}}
\def\test@subnexp@star#1{\@ifnextchar_{\isol@subnexp@star#1}{\test@exp@star#1}}
\def\test@exp@star#1{\@ifnextchar^{\isol@expnsub@star#1}{\mathbin{#1}}}
\def\isol@subnexp@star#1_#2{\@ifnextchar^{\eval@subnexp@star#1_#2}{\mathbin{\operatorname*{#1}_{#2}}}}
\def\eval@subnexp@star#1_#2^#3{\mathbin{\operatorname*{#1}_{#2}^{#3}}}
\def\isol@expnsub@star#1^#2{\@ifnextchar_{\eval@expnsub@star#1^#2}{\mathbin{\operatorname*{#1}^{#2}}}}
\def\eval@expnsub@star#1^#2_#3{\mathbin{\operatorname*{#1}_{#3}^{#2}}}
\def\declaremathbinop@nostar#1#2{\def#1{\test@subnexp@nostar#2}}
\def\test@subnexp@nostar#1{\@ifnextchar_{\isol@subnexp@nostar#1}{\test@exp@nostar#1}}
\def\test@exp@nostar#1{\@ifnextchar^{\isol@expnsub@nostar#1}{\mathbin{#1}}}
\def\isol@subnexp@nostar#1_#2{\@ifnextchar^{\eval@subnexp@nostar#1_#2}{\mathbin{\underset{#2}{#1}}}}
\def\eval@subnexp@nostar#1_#2^#3{\mathbin{\overset{#3}{\underset{#2}{#1}}}}
\def\isol@expnsub@nostar#1^#2{\@ifnextchar_{\eval@expnsub@nostar#1^#2}{\mathbin{\overset{#2}{#1}}}}
\def\eval@expnsub@nostar#1^#2_#3{\mathbin{\overset{#2}{\underset{#3}{#1}}}}
\makeatother

\let\OLDtimes\times
\DeclareMathBinOp*{\times}{\OLDtimes}
\let\OLDamalg\amalg
\DeclareMathBinOp*{\amalg}{\OLDamalg}
\let\OLDotimes\otimes
\DeclareMathBinOp*{\otimes}{\OLDotimes}
\let\OLDwedge\wedge
\DeclareMathBinOp*{\wedge}{\OLDwedge}

\makeatletter
\def\DeclareArrow#1#2{\def#1{\test@subnexp#2}}
\def\test@subnexp#1{\@ifnextchar_{\isol@subnexp#1}{\test@exp#1}}
\def\test@exp#1{\@ifnextchar^{\isol@expnsub#1}{#1}}
\def\isol@subnexp#1_#2{\@ifnextchar^{\eval@subnexp#1_#2}{\underset{#2}{#1}}}
\def\eval@subnexp#1_#2^#3{\underset{#2}{\overset{#3}{#1}}}
\def\isol@expnsub#1^#2{\@ifnextchar_{\eval@expnsub#1^#2}{\overset{#2}{#1}}}
\def\eval@expnsub#1^#2_#3{\overset{#2}{\underset{#3}{#1}}}
\makeatother

\let\OLDto\to
\DeclareArrow{\to}{\OLDto}
\DeclareArrow{\from}{\leftarrow}\newcommand{\A}{\mathbb{A}}
\newcommand{\T}{\mathbb{T}}
\newcommand{\ev}{\operatorname{ev}}
\newcommand{\op}{^{\mathrm{op}}}
\newcommand{\sSets}{\mathbf{sSets}}
\newcommand{\presh}{\operatorname{\mathcal{P}}}
\newcommand{\dAff}{\mathbf{dAff}}
\newcommand{\sCAlg}{\mathbf{sCAlg}}
\newcommand{\Perf}{\mathbf{Perf}}
\newcommand{\Map}{\operatorname{Map}}
\DeclareMathOperator*{\colim}{colim}
\newcommand{\B}{\operatorname{B}}
\newcommand{\Indu}[1]{\Ind^{\mathbb{#1}}}
\newcommand{\Prou}[1]{\Pro^{\mathbb{#1}}}
\newcommand{\Tateu}[1]{\Tate^{\mathbb{#1}}}
\newcommand{\inftyCatu}[1]{\inftyCat^{\mathbb{#1}}}
\newcommand{\PresLeftu}[1]{\mathbf{Pr}^{\mathrm{L,}\mathbb{#1}}_\infty}
\newcommandx*{\timesunder}[5][1={},2={},3=-2pt,4=0pt,5=0mm,usedefault]{\times_{\makebox[#5]{\raisebox{#3}{\ensuremath{\scriptstyle #1}}}}^{\makebox[#5]{\raisebox{#4}{\ensuremath{\scriptstyle #2}}}}}
\newcommand{\Gm}{\mathbb{G}_m}
\newcommandx*{\el}[4][2=\eldebutpardefaut,4={,}]{#1_{#2}#4\dots#4#1_{#3}}
\newcommand{\comma}[2]{\ensuremath \mathchoice {\raisebox{4pt}{$\displaystyle #1 $} \raisebox{2pt}{$\displaystyle / $} \displaystyle \hspace{-1pt}{#2}}{\raisebox{2pt}{$\textstyle #1 $} \raisebox{1pt}{$\textstyle / $} \textstyle \hspace{-1pt}{#2}}{\raisebox{2pt}{$\scriptstyle #1 $} \raisebox{1pt}{$\scriptstyle / $} \scriptstyle \hspace{-1pt}{#2}}{\raisebox{2pt}{$\scriptscriptstyle #1 $} \raisebox{1pt}{$\scriptscriptstyle / $} \scriptscriptstyle \hspace{-1pt}{#2}}}
\newcommandx*{\cart}[3][1=1,2=5,3=10,usedefault]{\ar@{-}[]+D+<#3pt,#1pt>+<#2pt,0pt>;[]+D+<#3pt,-#3pt>+<#2pt,#1pt> \ar@{-}[]+D+<0pt,-#3pt>+<#2pt,#1pt>;[]+D+<#3pt,-#3pt>+<#2pt,#1pt>}
\newcommandx*{\cocart}[3][1=-1,2=8,3=10,usedefault]{\ar@{-}[]+U+<-#2pt,-#1pt>;[]+U+<-#2pt,-#1pt>+<0pt,#3pt> \ar@{-}[]+U+<-#2pt,-#1pt>;[]+U+<-#2pt,-#1pt>+<-#3pt,0pt>}
\newcommand{\Z}{\mathbb{Z}}
\newcommand{\Q}{\mathbb{Q}}
\newcommand{\dual}[1]{{#1}^{\vee}}
\newcommand{\Cc}{\mathcal{C}}
\newcommand{\Dd}{\mathcal{D}}
\newcommand{\Ind}{\operatorname{\mathbf{Ind}}}
\newcommand{\Pro}{\operatorname{\mathbf{Pro}}}
\newcommand{\Tate}{\operatorname{\mathbf{Tate}}}
\newcommand{\quot}[2]{\ensuremath \mathchoice {\displaystyle #1 \raisebox{-2pt}{$\displaystyle \hspace{-1pt}{/} $} \raisebox{-4pt}{$\displaystyle \hspace{-1pt}{#2}$}}{\textstyle #1 \raisebox{-1pt}{$\textstyle \hspace{-1pt}{/} $} \raisebox{-2pt}{$\textstyle \hspace{-1pt}{#2}$}}{\scriptstyle #1 \raisebox{-1pt}{$\scriptstyle \hspace{-1pt}{/} $} \raisebox{-2pt}{$\scriptstyle \hspace{-1pt}{#2}$}}{\scriptscriptstyle #1 \raisebox{-1pt}{$\scriptscriptstyle \hspace{-1pt}{/} $} \raisebox{-2pt}{$\scriptscriptstyle \hspace{-1pt}{#2}$}}}
\newcommand{\mymatrix}{\shorthandoff{;:!?} \xymatrix}
\newcommand{\id}{\operatorname{id}}
\newcommand{\noloc}{\,:}
\newcommand{\loccit}{\emph{loc. cit.} }
\newcommand{\Oo}{\mathcal{O}}
\newcommand{\inftyCat}{\mathbf{Cat}_\infty}
\newcommand{\dSt}{\mathbf{dSt}}
\newcommand{\Homint}{\operatorname{\underline{Hom}}}
\newcommand{\Fct}{\operatorname{Fct}}
\newcommand{\Ee}{\mathcal E}
\newcommand{\homol}{\mathrm H}
\newcommand{\Det}{\operatorname{Det}}
\newcommand{\Spec}{\operatorname{Spec}}
\newcommand{\eldebutpardefaut}{1}

\usepackage{geometry}
\setlist[enumerate]{label=\emph{(\roman*)},ref=\emph{(\roman*)}}

\entrymodifiers={!!<0pt,0.7ex>+}

\title{Tate objects in stable $(\infty,1)$-categories}
\author{Benjamin Hennion\footnote{hennion@mpim-bonn.mpg.de, Max-Planck Institut für Mathematik, Vivatsgasse 7, Bonn, Germany}}
\date{\today}

\newlist{assertions}{enumerate}{1}
\setlist[assertions]{label={(\alph*)}, ref={assertion (\alph*)}}
\newlist{assumptions}{enumerate}{1}
\setlist[assumptions]{label={(\roman*)}, ref={assumption (\roman*)}}
\newlist{disjunction}{enumerate}{1}
\setlist[disjunction]{label={(\arabic*)}, ref={case (\arabic*)}}

\begin{document}

\selectlanguage{english}
\maketitle

\begin{abstract}
Tate objects have been studied by many authors. They allow us to deal with infinite dimensional spaces by identifying some more structure. In this article, we set up the theory of Tate objects in stable $(\infty,1)$-categories, while the literature only treats with exact categories. We will prove the main properties expected from Tate objects.
This new setting includes several useful examples: Tate objects in the category of spectra for instance, or in the derived category of a derived algebraic object -- which can be thought as structured infinite dimensional vector bundle in derived setting.
\end{abstract}
{\bf AMS class:} 18F25 18G55

\tableofcontents

\section*{Introduction}
\addcontentsline{toc}{section}{Introduction}%

Tate vector spaces have been used by many to deal with infinite dimensional spaces. Identifying some structure allows us to define a well-behaved duality on infinite dimensional spaces.
They were studied by several authors, including Lefschetz, Beilinson, Drinfeld and more recently Osipov and Zhu, Previdi, Saito, and Bräunling, Gröchenig and Wolfson.

In \cite{drinfeld:tate}, Drinfeld describes them the following way. Let us fix a field $k$. Let $V$ be a vector space, which we see as a discrete topological space. Its topological dual $\dual V$ is then what is called a linearly compact vector space. 
A Tate vector space is a topological space of the form $V \oplus \dual W$ where both $V$ and $W$ are discrete topological vector spaces.
The first example of such a Tate vector space is the field of Laurent series $k(\!( t )\!) \simeq k[\![ t ]\!] \oplus t^{-1} k[t^{-1}]$.
For any Tate vector space $X$, we then have $\dual{(\dual X)} \simeq X$.

Bräunling, Gröchenig and Wolfson then generalised this idea to any exact category, instead of vector spaces -- see \cite{bgw:tate}. Let $\Cc$ be an exact category -- whose objects will play the role of finite dimensional vector spaces. The category of ind-object in $\Cc$ is then an analogue to infinite dimensional spaces. The dual of an ind-object is naturally a pro-object. They define an elementary Tate object in $\Cc$ as an extension -- in a suitable category -- of a pro-object by an ind-object. An object $X$ is an elementary Tate object if it fits in an exact sequence
\[
X^p \to X \to X^i
\]
where $X^p$ is a pro-object and $X^i$ is an ind-object.
A Tate object is then a retract of an elementary Tate object.
Amongst examples of Tate objects is the field $\Q_p$, sitting in the exact sequence
\[
\Z_p \to \Q_p \to \quot{\Q_p}{\Z_p}
\]
of abelian groups. This example fits in the formalism of \cite{bgw:tate}.

In this article, we define and study Tate objects in stable and idempotent complete $(\infty,1)$-categories.
The context of higher categories allows us to talk about Tate objects in spectra for instance, or in derived algebraic geometry (see below).
If $\Cc$ is such a stable $(\infty,1)$-category, then the categories of ind- and pro-objects in $\Cc$ -- denoted by $\Ind(\Cc)$ and $\Pro(\Cc)$ -- are also stable.
We also consider the $(\infty,1)$-category of pro-ind-objects in $\Cc$, denoted by $\Pro \Ind(\Cc)$.
This allows us the following
\begin{dfintro}
Let $\Tate_\mathrm{el}(\Cc)$ denote the smallest full subcategory of $\Pro\Ind(\Cc)$ containing both the essential images of $\Ind(\Cc)$ and $\Pro(\Cc)$ and stable (by extension).
The category $\Tate(\Cc)$ of Tate objects is the idempotent completion of $\Tate_\mathrm{el}(\Cc)$.
\end{dfintro}

This definition is a priori different to the one we gave above. We will see below (see \autoref{introlattices}) that they coincide. For now, let us identify a universal property of $\Tate(\Cc)$.
\begin{thmintro}[see \autoref{univpropT0}]\label{introunivprop}
Let $\Cc$ be a stable and idempotent complete $(\infty,1)$-category. For any stable and idempotent complete $(\infty,1)$-category $\Dd$ and any commutative diagram of exact functors
\[
\mymatrix{
\Cc \ar[r]^-i \ar[d]_j & \Ind(\Cc) \ar[d]^f \\ \Pro(\Cc) \ar[r]_-g & \Dd
}
\]
such that $i$ and $j$ are the canonical embeddings, and such that $f$ preserves filtered colimits and $g$ preserves cofiltered limits there exists an essentially unique exact functor $\Tate(\Cc) \to \Dd$ through which both $f$ and $g$ factor.
\end{thmintro}

In particular, this theorem allows us to identify $\Tate(\Cc)$ with the smallest full subcategory of $\Ind \Pro(\Cc)$ containing $\Ind(\Cc)$ and $\Pro(\Cc)$, and stable and idempotent complete.
The proof of \autoref{introunivprop} uses tools of $(\infty,1)$-category theory developed in \cite{lurie:htt}.

We then focus on the properties Tate objects are supposed to satisfy. Note that the first item makes the two definitions of elementary Tate objects we gave coincide.
\begin{thmintro}\label{introlattices}
Let $\Cc$ be a small, stable and idempotent complete $(\infty,1)$-category and let $X$ be an elementary Tate object in $\Cc$.
\begin{enumerate}
\item There exists an exact sequence in $\Pro \Ind(\Cc)$
\[
X^p \to X \to X^i
\]
where $X^p \in \Pro(\Cc)$ and $X^i \in \Ind(\Cc)$ (see \autoref{latticesexist}). Such an exact sequence is called a lattice of $X$. Lattices of $X$ form an $(\infty,1)$-category.
\item The category of lattices of $X$ is small and both filtered and cofiltered, and we have
\[
X \simeq \lim_{X^\bullet} X^i \in \Pro \Ind(\Cc)
\hspace{5mm} \text{and} \hspace{5mm}
X \simeq \colim_{X^\bullet} X^p \in \Ind \Pro(\Cc)
\]
where $X^\bullet =(X^p \to X \to X^i)$ runs through the category of lattices of $X$ (see \autoref{colimoflattices}).
\end{enumerate}
\end{thmintro}
To prove \autoref{introlattices}, we first show that the category of pro-ind-objects which admit a lattice is stable by extension. This proves the first item.
Moreover, a map of between lattices of $X$ is essentially determined by an object of $\Cc$. This implies the smallness of the category of lattices. 
To show that it is furthermore filtered, we build some kind of enveloping lattice, sitting under a finite family of lattices of $X$.
Note that this is an $\infty$-categorical version of the main theorem of \cite{bgw:tate}.
Finally, to prove that $X$ is the limit of its lattices, we identify the canonical embedding $\Tate(\Cc) \to \Pro \Ind (\Cc)$ with the right Kan extension of a functor mapping a lattice $X^p \to X \to X^i$ to $X^i \in \Ind(\Cc)$.

Similarly to K-theory of exact categories, any stable and idempotent complete $(\infty,1)$-category $\Cc$ has a non-connective K-theory spectrum $\mathbb K(\Cc)$ -- see \cite{bgt:characterisationk}.
Our next result gives an $\infty$-categorical analogue to a delooping result of Saito in \cite{saito:deloop}, first conjectured by Kapranov and Previdi.
\begin{thmintro}[see \autoref{ktheorysusp}]\label{introksusp}
Let $\Cc$ be a stable and idempotent complete $(\infty,1)$-category. The non-connective K-theory of $\Tate(\Cc)$ is the suspension of that of $\Cc$:
\[
\mathbb K(\Tate(\Cc)) \simeq \Sigma \mathbb K(\Cc)
\]
\end{thmintro}
The proof of the above theorem is strongly inspired by that of Saito in the case of exact categories. We first show that the quotients
\[
\quot{\Ind(\Cc)}{\Cc} \hspace{5mm} \text{and} \hspace{5mm} \quot{\Tate(\Cc)}{\Pro(\Cc)}
\]
in stable and idempotent complete $(\infty,1)$-categories are equivalent. The non-connective K-theory functor then preserves exact sequences of stable and idempotent complete $(\infty,1)$-categories -- see \cite{bgt:characterisationk}.

Our last result concerns the definition of a determinant map from the (algebraic) connective K-theory of Tate objects into the suspension of the Picard moduli space. This question has been studied in the context of exact categories by Osipov and Zhu in \cite{osipovzhu:categorical}.
Let us denote by $\mathrm K^{\Perf}$ the moduli space of K-theory of perfect complexes -- ie the presheaf on affine schemes
\[
A \mapsto \mathrm K^{\Perf}(A) = \mathrm K(\Perf(A))
\]
We will prove the following statement in the context of derived algebraic geometry\footnote{see below.}. 
\begin{thmintro}[see \autoref{detclass}]\label{introdet}
Let $\mathrm K^{\Tate}$ denote the presheaf $A \mapsto \mathrm K(\Tate(\Perf(A)))$. Let $\mathrm K(\Gm,2)$ denote the Eilenberg-Maclane classifying stack.
The determinant $\Det \colon \mathrm K^{\Perf} \to \B\Gm$ induces a morphism
\[
\Det \colon \mathrm K^{\Tate} \to \mathrm K(\Gm,2)
\]
In particular, any Tate object $E$ over some $X$ induces a determinantal class $[\Det_E] \in \homol^2(X,\Oo_X^{\times})$.
\end{thmintro}
We have an exact sequence
\[
\B \mathrm K^{\Perf} \to \mathrm K^{\Tate} \to \mathrm K^{\Tate}_0
\]
To prove \autoref{introdet}, we show that $\mathrm K^{\Tate}_0$ vanishes Nisnevich locally. We can then use the determinant map $\mathrm K^{\Perf} \to \B \Gm$ to define the announced map.

Note that the above theorem applies to any additive invariant, not only to the determinant (see the last section).

\paragraph*{Applications:}
Tate objects naturally appear when considering local fields: the field of Laurent series $k(\!(t)\!)$ is a Tate vector space of $k$.
The first application the author has in mind concerns the study of formal loop spaces, as defined in \cite{kapranovvasserot:loop1}. To any scheme $X$ of finite type, we associate its formal loop space: roughly speaking, it is an ind-pro-scheme $\mathcal L^1(X)$ representing the functor $A \mapsto X(A(\!(t)\!))$.

Examples of stable $(\infty,1)$-categories naturally appear in derived algebraic geometry. Derived algebraic geometry allows us to study ill-behaved geometric situations. The most emblematic examples are the study of non-generic intersections and of quotients by a wild action.
In this context, the category of quasi-coherent complexes becomes a central object. This category is actually a stable and idempotent complete $(\infty,1)$-category. This core example of such a category motivates the results of this article.

For instance, in \cite{hennion:floops}, the author develops a higher dimensional analogue $\mathcal L^d(X)$ of the formal loop space. This new geometrical object is actually a derived stack. In this context, the derived category of $\mathcal L^d(X)$ is a stable $(\infty,1)$-category.
Exact categories are not enough here. In \cite{hennion:floops}, it is proven that in some cases the tangent of $\mathcal L^d(X)$ is a Tate module.
Our \autoref{introdet} then defines a determinantal anomaly for those higher dimensional formal loop spaces, generalising a result of \cite{kapranovvasserot:loop2}.
Moreover, the nice properties of Tate objects together with derived symplectic geometry allows us to define symplectic structure on infinite dimension algebraic objects.

The construction we provide $\Cc \mapsto \Tate(\Cc)$ can of course be iterated. We would then obtain some categories $\Tate^n(\Cc)$ for any integer $n$. Those kind of construction appeared with Beilinson's adèles and local fields in several variables.

Another source of example is topology. If $X$ is a space, then the category of spectra over $X$ is a stable $(\infty,1)$-category. The \autoref{introksusp} then gives a shifted version of Waldhausen's K-theory of $X$. Again, this example could not be studied using only exact categories.

\paragraph*{Related work:}
Literature on Tate objects is flourishing. Let us cite here the work of Drinfeld \cite{drinfeld:tate}, Previdi \cite{previdi:thesis}, Saito \cite{saito:deloop}, Osipov and Zhu \cite{osipovzhu:categorical} and more recently Bräunling, Gröchenig and Wolfson \cite{bgw:tate}. The author has also been told that Barwick, Gröchenig and Wolfson are currently working on a theory of Tate objects in exact $(\infty,1)$-categories.

\paragraph*{Acknowledgements:}
The author would like to thank Bertrand Toën, Marco Robalo, Michael Gröchenig and Damien Calaque for the many discussions we had about the content of this article.

\section{Preliminaries}

This first section contains $\infty$-categorical preliminaries. Most of the content comes from \cite{lurie:htt}. We will also define here a few notations.

\paragraph*{Notations:}
Throughout this article, we will fix two universes $\mathbb U \in \mathbb V$.

Let us first set a few notations, borrowed from \cite{lurie:htt}.
\begin{itemize}
\item We will denote by $\inftyCatu U$ the $(\infty,1)$-category of $\mathbb U$-small $(\infty,1)$-categories -- see \cite[3.0.0.1]{lurie:htt};
\item Let $\PresLeftu U$ denote the $(\infty,1)$-category of $\mathbb U$-presentable (and thus $\mathbb V$-small) $(\infty,1)$-categories with left adjoint functors -- see \cite[5.5.3.1]{lurie:htt};
\item The symbol $\sSets$ will denote the $(\infty,1)$-category of $\mathbb U$-small spaces;
\item For any $(\infty,1)$-categories $\Cc$ and $\Dd$ we will write $\Fct(\Cc,\Dd)$ for the $(\infty,1)$-category of functors from $\Cc$ to $\Dd$ (see \cite[1.2.7.3]{lurie:htt}). The category of presheaves will be denoted $\presh(\Cc) = \Fct(\Cc\op, \sSets)$.
\item For any $(\infty,1)$-category $\Cc$ and any objects $c$ and $d$ in $\Cc$, we will denote by $\Map_{\Cc}(c,d)$ the space of maps from $c$ to $d$.
\end{itemize}

The following theorem is a concatenation of results from Lurie.
\begin{thm}[Lurie]\label{indu-thm}
Let $\Cc$ be a $\mathbb V$-small $(\infty,1)$-category.
There is an $(\infty,1)$-category $\Indu U(\Cc)$ and a functor $j \colon \Cc \to \Indu U(\Cc)$ such that
\begin{enumerate}
\item The $(\infty,1)$-category $\Indu U(\Cc)$ is $\mathbb V$-small;
\item The $(\infty,1)$-category $\Indu U(\Cc)$ admits $\mathbb U$-small filtered colimits and is generated by $\mathbb U$-small filtered colimits of objects in $j(\Cc)$;
\item The functor $j$ is fully faithful and preserves finite limits and finite colimits which exist in $\Cc$;
\item For any $c \in \Cc$, its image $j(c)$ is $\mathbb U$-small compact in $\Indu U (\Cc)$;
\item For every $(\infty,1)$-category $\Dd$ with every $\mathbb U$-small filtered colimits, the functor $j$ induces an equivalence
\[
 \Fct^{\mathbb U\mathrm{-c}}(\Indu U(\Cc), \Dd) \to^\sim \Fct(\Cc,\Dd)
\]
where $\Fct^{\mathbb U \mathrm{-c}}(\Indu U(\Cc), \Dd)$ denote the full subcategory of $\Fct(\Indu U(\Cc),\Dd)$ spanned by functors preserving $\mathbb U$-small filtered colimits.
\item If $\Cc$ is $\mathbb U$-small and admits all finite colimits then $\Indu U(\Cc)$ is $\mathbb U$-presentable;
\end{enumerate}
\end{thm}

\begin{proof}
Let us use the notations of \cite[5.3.6.2]{lurie:htt}. Let $\mathcal K$ denote the collection of $\mathbb U$-small filtered simplicial sets. We then set $\Indu U (\Cc) = \presh^\mathcal K _\emptyset(\Cc)$. Recall that $\Indu U(\Cc)$ is then the full subcategory of $\presh(\Cc)$ generated by $\mathbb U$-small filtered colimits of diagrams in $\Cc$.
It satisfies the required properties because of \loccit 5.3.6.2 and 5.5.1.1. We also need tiny modifications of the proofs of \loccit 5.3.5.14 and 5.3.5.5. The last item is proved in \cite[6.3.1.10]{lurie:halg}.
\end{proof}

Lurie proved in \cite[5.3.5.15]{lurie:htt} that any map $c \to d \in \Indu U (\Cc)$ is a colimit of a $\mathbb U$-small filtered diagram $K \to \Fct(\Delta^1, \Cc)$. We will need afterwards the following small refinement of this statement, inspired by \cite[3.9]{bgw:tate}
\begin{prop}[Strictification of morphisms]\label{strictification}
Let $\Cc$ be a $\mathbb V$-small $(\infty,1)$-category. Let $f \colon c \to d$ be a morphism in $\Indu U(\Cc)$. Let also $\bar c \colon K \to \Cc$ and $\bar d \colon L \to \Cc$ be $\mathbb U$-small filtered diagrams of whom respectively $c$ and $d$ are colimits in $\Indu U(\Cc)$.
There exists a $\mathbb U$-small filtered diagram $\bar f \colon J \to \Fct(\Delta^1, \Cc)$ and a commutative diagram
\[
\mymatrix{
K \ar[d]^{\bar c} & J \ar[r]^{p_L} \ar[l]_{p_K} \ar[d]^{\bar f} & L \ar[d]^{\bar d} \\
\Cc & \Fct(\Delta^1, \Cc) \ar[r]^-{\ev_1} \ar[l]_-{\ev_0} & \Cc
}
\]
such that both maps $p_K$ and $p_L$ are cofinal, and such that $f$ is the colimit of $\bar f$.
\end{prop}

\begin{proof}
Using \cite[4.3.2.14]{lurie:htt} we can assume that both $K$ and $L$ are filtered partially ordered sets.
Let us denote by $J'$ the fibre product
\[
\mymatrix{
J' \cart \ar[d] \ar[rr] && \quot{\Fct(\Delta^1, \Indu U(\Cc))}{f} \ar[d] \\
\displaystyle K \times_\Cc \Fct(\Delta^1, \Cc) \times_\Cc L \ar[r] & \Fct(\Delta^1, \Cc) \ar[r] & \Fct(\Delta^1, \Indu U(\Cc))
}
\]
Let us first prove that $J'$ is filtered.
Let $P$ be a partially ordered finite set and $P^\triangleright$ denote the partially ordered set $P \cup \{\infty\}$, where $\infty$ is a maximal element.
A morphism $P \to J'$ is the datum of a commutative diagram
\[
\mymatrix@!0@R=6mm@C=10mm{
&& P \times \{0 \} \ar[rrr]^\kappa \ar[ddr] &&& K \ar[rdd]^{\bar c} & \\
&P \times \{1\} \ar[rrr]|!{[ur];[drr]}\hole^(0.7)\lambda \ar[rrd] &&& L \ar[drr]^(0.4){\bar d} && \\
&&& P \times \Delta^1 \ar[rrr]_\psi \ar[dd] &&& \Cc \ar[dd] \\
\\
\{\infty\} \times \Delta^1 \ar[rrr] &&& P^\triangleright \times \Delta^1 \ar[rrr] &&& \Indu U(\Cc)
}
\]
Let us denote by $P_+$ the partially ordered set $P \cup \{ + \}$ where $+$ is a maximal element. Because $K$ is filtered, the map $\kappa$ extends to a morphism $\kappa' \colon P_+ \times \{0\} \to K$.
There exists $l \in L$ such that the induced map $\bar c(\kappa'(+)) \to c \to d$ factors through $\bar d(l) \to d$. Since $L$ is filtered, there is a map $\lambda' \colon P_+ \times \{1\} \to L$ extending $\lambda$.
We can moreover chose $\lambda'(+)$ greater than $l$ (ie with a map $l \to \lambda'(+)$ in $L$).
Using the map $\bar c(\kappa'(+)) \to \bar d(l) \to \bar d(\lambda'(+))$, we get a morphism $\psi' \colon P_+ \times \Delta^1 \to \Cc$ extending $\psi$, which by construction extends to $P_+^\triangleright \times \Delta^1$ -- where we set $\infty \geq +$.
This defines a morphism $P_+ \to J'$, proving that $J'$ is filtered.
Using \cite[4.3.2.14]{lurie:htt} we define $J$ to be a filtered partially ordered set with a cofinal map $J \to J'$.
Proving that the maps $J \to K$ and $J \to L$ are cofinal is now straightforward.
This also implies that the induced diagram $\bar f \colon J \to \Fct(\Delta^1, \Cc)$ has colimit $f$ in $\Indu U(\Cc)$.
\end{proof}

\begin{rmq}
Note that when $\Cc$ admits finite colimits then the category $\Indu U(\Cc)$ embeds in the $\mathbb V$-presentable category $\Indu V(\Cc)$.
\end{rmq}

\begin{df}
Let $\Cc$ be a $\mathbb V$-small $\infty$-category. We define $\Prou U(\Cc)$ as the $(\infty,1)$-category
\[
\Prou U(\Cc) = \left( \Indu U(\Cc\op) \right) \op
\]
It satisfies properties dual to those of $\Indu U(\Cc)$.
\end{df}

The following lemma is a direct consequence of results from Lurie's \cite{lurie:htt}.
\begin{lem}[Stable envelop]\label{tatification}
Let $\Cc$ be a $\mathbb V$-small pointed category with all suspensions. Let us assume that the suspension functor $\Cc \to \Cc$ is an equivalence.
There exists an $(\infty,1)$-category $\Cc^\mathrm{st}$ with a map $j \colon \Cc \to \Cc^\mathrm{st}$ such that
\begin{enumerate}
\item The category $\Cc^\mathrm{st}$ is $\mathbb V$-small and stable.
\item The functor $j$ is fully faithful and preserves all limits and finite colimits which exist in $\Cc$.
\item For any stable $(\infty,1)$-category $\Dd$ the induced map
\[
\Fct^\mathrm{ex}(\Cc^\mathrm{st},\Dd) \to \Fct^\mathrm{lex}(\Cc,\Dd)
\]
between exact functors and left exact functors is an equivalence.
\item For any stable category $\Dd$ with a fully faithful functor $\Cc \to \Dd$ preserving finite colimits and limits which exist in $\Cc$, the smallest stable subcategory of $\Dd$ containing the image of $\Cc$ is equivalent to $\Cc^\mathrm{st}$.
\end{enumerate}
\end{lem}
 
In the proof of the above lemma, we will need the notation:
\begin{df}[{see \cite[1.2.8.4]{lurie:htt}}]
Let $K$ be a simplicial set. We will denote by $K^{\triangleright}$ the simplicial set obtained from $K$ by formally adding a final object. This final object will be called the cone point of $K^\triangleright$.
\end{df}

\begin{proof}
Let us denote by $K$ the simplicial set corresponding to a diagram $\bullet \from \bullet \to \bullet$. Let $\mathcal R$ denote the collection of all cocartesian diagrams $K^\triangleright \to \Cc$ and the zero $\emptyset^\triangleright = \bullet \to \Cc$ in $\Cc$.
We then set $\Cc^\mathrm{st} = \presh^{\{K,\emptyset\}}_{\mathcal R}(\Cc)$ using the notation of \cite[5.3.6.2]{lurie:htt}. Note that \emph{(iii)} is proven in \emph{loc. cit.}.
The category $\Cc^\mathrm{st}$ is pointed and it comes with two natural fully faithful maps
\[
\mymatrix{
\Cc \ar[r]^-j & \Cc^\mathrm{st} \ar[r] & \presh(\Cc)
}
\]
whose composite is the Yoneda functor and therefore preserves limits which exist in $\Cc$. It follows that $j$ also preserves those limits.
By definition, the functor $j$ preserves finite colimits which exist in $\Cc$.

Any object of $\Cc^\mathrm{st}$ is a finite colimit of objects in $\Cc$. Its suspension is therefore the colimit of the suspensions of those objects.
The suspension functor $\Cc^\mathrm{st} \to \Cc^\mathrm{st}$ is thus an equivalence.
Corollary \cite[1.4.2.27]{lurie:halg} implies that $\Cc^\mathrm{st}$ is stable.

We now focus on the assertion \emph{(iv)}. Let $f \colon \Cc \to \Dd$ be as required. Because of the third point, there is an essentially unique functor $g \colon \Cc^\mathrm{st} \to \Dd$ lifting $f$. Every object in $\Cc^\mathrm{st}$ can be written as both a colimit and a limit of objects of $\Cc$. It follows that $g$ is fully faithful and then that $\Cc^\mathrm{st}$ contains the smallest full and stable subcategory $\Dd'$ of $\Dd$ extending $\Cc$. There is also a universal map $\Cc^\mathrm{st} \to \Dd'$ which is easily seen to be an inverse to the inclusion.
\end{proof}

\begin{lem}\label{indandproareinc}
Let $\Cc$ be an idempotent complete $\mathbb V$-small $(\infty,1)$-category. We consider the natural embeddings $i \colon \Prou U (\Cc) \to \Prou U \Indu U(\Cc)$ and $j \colon \Indu U(\Cc) \to \Prou U \Indu U(\Cc)$. We will also denote by $k$ the embedding $\Cc \to \Prou U \Indu U(\Cc)$.
If an object of $\Prou U \Indu U(\Cc)$ is in both the essential images of $i$ and $j$, then it is in the essential image of $k$.
\end{lem}

\begin{proof}
Let $x \in \Indu U(\Cc)$.
Let us assume there exists a pro-object $y \in \Prou U(\Cc)$ and an equivalence $f \colon x \to y$. Let $\bar y \colon K \op \to \Cc$ be a cofiltered diagram of whom $y$ is a limit in $\Prou U(\Cc)$. The equivalence $f$ induces a morphism from the constant diagram $x \colon K\op \to \Indu U(\Cc)$ to $\bar y \colon K \op \to \Indu U(\Cc)$.
An inverse $g \colon y \to x$ of $f$ then induces a map $y_k = \bar y(k) \to x$ for some $k \in K$ such that the composite morphism $x \to y_k \to x$ is homotopic to the identity. Idempotent completeness and \cite[5.4.2.4]{lurie:htt} finish the proof.
\end{proof}

\begin{df}
Let $\inftyCat^{\mathbb V\mathrm{,st}}$ denote the subcategory of $\inftyCatu V$ spanned by stable categories with exact functors between them -- see \cite[1.1.4]{lurie:halg}.
Let $\inftyCat^{\mathbb V\mathrm{,st,id}}$ denote the full subcategory of $\inftyCat^{\mathbb V\mathrm{,st}}$ spanned by idempotent complete stable categories.
\end{df}

\section{Tate objects}
In this subsection we define the category of Tate objects in a stable $(\infty,1)$-category.
\begin{df}\label{dftate}
Let $\Cc$ be a $\mathbb V$-small stable $(\infty,1)$-category. We define the category $\Tateu U_0(\Cc)$ of pure Tate objects in $\Cc$ as the full sub-category of $\Prou U\Indu U(\Cc)$ spanned by the images of $\Indu U(\Cc)$ and $\Prou U(\Cc)$ through the canonical functors.
The category $\Tateu U_0(\Cc)$ obviously satisfies the conditions of \autoref{tatification} and we define the category $\Tateu U_\mathrm{el}(\Cc)$ of elementary Tate objects in $\Cc$ as the stable envelop
\[
\Tateu U_\mathrm{el}(\Cc) = \left( \Tateu U_0(\Cc) \right) ^{\mathrm{st}}
\]
We also define the category $\Tateu U(\Cc)$ of Tate objects in $\Cc$ as the idempotent completion of $\Tateu U_\mathrm{el}(\Cc)$.
We have fully faithful exact functors between stable $(\infty,1)$-categories
\[
\Tateu U_\mathrm{el}(\Cc) \to \Tateu U(\Cc) \to \Prou U\Indu U(\Cc)
\]
\end{df}

\begin{rmq}
It follows from \autoref{tatification} that $\Tateu U_\mathrm{el}(\Cc)$ (resp. $\Tateu U(\Cc)$) is the smallest stable (resp. stable and idempotent complete) subcategory of $\Prou U \Indu U(\Cc)$ containing both the essential images of $\Indu U(\Cc)$ and $\Prou U(\Cc)$.
We will see (\autoref{cortateinindpro}) that the same holds in $\Indu U \Prou U(\Cc)$ instead of $\Prou U \Indu U(\Cc)$.
\end{rmq}

\begin{ex}
Let $\Cc$ be the category of perfect complexes over a field $k$. There is a natural incarnation of $k(\!(t)\!)$ in $\Tateu U(\Perf_k)$, given by the isomorphism (of vector spaces)
\[
k(\!(t)\!) \simeq \lim_p \colim_n t^{-n} \quot{k[t]}{t^{p+n}}
\]
To see it actually lives in Tate objects, one can write the isomorphism $k(\!(t)\!) \simeq k[\![t]\!] \oplus t^{-1} k[t^{-1}]$
and translate it in terms of objects in $\Prou U \Indu U(\Perf_k)$.
It follows that $\lim_p \colim_n t^{-n} \quot{k[t]}{t^{p+n}}$ indeed lies in $\Tateu U(\Perf_k)$.
We will see in \autoref{sectionlattices} that any Tate object can be written as such an extension, of an ind-object by a pro-object.

The above example does not require $(\infty,1)$-categories to work. Although, it leads to the following generalisation. We can see $k(\!(t)\!)$ as the ring of functions on the punctured formal neighbourhood $\widehat \A^1 \smallsetminus \{0\}$. Now considering the complex of derived global sections of the sheaf of functions on the punctured formal neighbourhood $\widehat \A^d \smallsetminus \{0\} = \Spec(k[\![\el{t}{d}]\!]) \smallsetminus \{0\}$ of dimension $d$. Computing its cohomology, we get
\[
\homol^n(\widehat \A^d \smallsetminus \{0\}, \Oo) \simeq
\begin{cases}
k[\![\el{t}{d}]\!] & \text{if } n = 0
\\
(\el{t}{d}[])^{-1}k[\el{t^{-1}}{d}] & \text{if } n = d-1
\\
0 & \text{else}
\end{cases}
\]
Hence we get $\mathbb R \Gamma(\widehat \A^d \smallsetminus \{0\}, \Oo) \simeq k[\![\el{t}{d}]\!] \oplus (\el{t}{d}[])k[\el{t^{-1}}{d}][1-d]$ (where $[1-d]$ is the shift by $1-d$). It is again the extension of a pro-object by an ind-object.
\end{ex}

\begin{rmq}
Note that $\Tateu U (\Cc)$ is $\mathbb V$-small.
The construction $\Tateu U(-)$ defines a functor 
\[
\inftyCat^{\mathbb V\mathrm{,st}} \to \inftyCat^{\mathbb V\mathrm{,st,id}}
\]
It comes with a fully faithful --- ie pointwise fully faithful --- natural transformation
\[
\mymatrix{
\inftyCat^{\mathbb V\mathrm{,st}} \ar[rr]^-{\Tateu U} \ar[d] && \inftyCat^{\mathbb V\mathrm{,st,id}} \ar[d] \ar@{=>}[dll] \\ \inftyCatu V \ar[rr]_{\Prou U \Indu U} && \inftyCatu V
}
\]
\end{rmq}

\begin{rmq}
We can immediately see that the functor $\Tateu U$ map any fully faithful and exact functor $\Cc \to \Dd$ between stable categories to a fully faithful (and exact) functor $\Tateu U(\Cc) \to \Tateu U(\Dd)$.
\end{rmq}

Let us now give a universal property for the category of pure Tate objects. The next theorem states that for any $(\infty,1)$-category $\Dd$ and any commutative diagram
\[
\mymatrix{
\Cc \ar[r] \ar[d] & \Indu U(\Cc) \ar[d]^-f \\ \Prou U(\Cc) \ar[r]_-g & \Dd
}
\]
such that $f$ preserves $\mathbb U$-small filtered colimits and $g$ preserves $\mathbb U$-small cofiltered limits there exists an essentially unique functor $\Tateu U_0(\Cc) \to \Dd$ such that $f$ and $g$ are respectively equivalent to the composite functors
\begin{align*}
&\Indu U(\Cc) \to \Tateu U_0(\Cc) \to \Dd \\
&\Prou U(\Cc) \to \Tateu U_0(\Cc) \to \Dd
\end{align*}
This universal property was discovered during a discussion with Michael Gröchenig, whom the author thanks greatly.
To state formally this property, let us fix some notations. Let $i \colon \Indu U(\Cc) \to \Tateu U_0(\Cc)$ and $p \colon \Prou U(\Cc) \to \Tateu U_0(\Cc)$ denote the canonical inclusions.
We will denote by $\Fct_t(\Tateu U_0(\Cc),\Dd)$ the full subcategory of $\Fct(\Tateu U_0(\Cc),\Dd)$ spanned by those functors $\xi$ such that
\begin{itemize}
\item The composite functor $\xi i$ maps filtered colimit diagrams to colimit diagrams.
\item The composite functor $\xi p$ maps cofiltered limit diagrams to limit diagrams.
\end{itemize}
Let also $\Fct_m(\Cc,\Dd)$ denote the category of functors $g \colon \Cc \to \Dd$ such that
\begin{itemize}
\item For any filtered diagram $K \to \Cc$, the composite diagram $K \to \Cc \to \Dd$ admits a colimit in $\Dd$.
\item For any cofiltered diagram $K\op \to \Cc$, the composite diagram $K\op \to \Cc \to \Dd$ admits a limit in $\Dd$.
\end{itemize}
\begin{thm}\label{univpropT0}
Let $\Cc$ be a $\mathbb V$-small stable $(\infty,1)$-category. For any $(\infty,1)$-category $\Dd$, the restriction functor induces an equivalence
\[
\mymatrix{
\Fct_t(\Tateu U_0(\Cc),\Dd) \ar[r] & \Fct_m(\Cc,\Dd)
}
\]
\end{thm}
\begin{proof}
Let us shorten the notations:
\[
\mathrm I \Cc = \Indu U(\Cc) \hspace{1cm} \mathrm P \Cc = \Prou U(\Cc) \hspace{1cm} \mathrm T = \Tateu U_0(\Cc) \hspace{1cm} \mathrm{PI} \Cc = \Prou U(\Indu U(\Cc))
\]
Recall that $\presh(\Dd)$ denotes the $(\infty,1)$-category of simplicial presheaves on $\Dd$.
The restriction functor $\Fct(\mathrm{PI} \Cc,\presh(\Dd)) \to \Fct(\mathrm T\Cc,\presh(\Dd))$ admits a left adjoint given by the left Kan extension. The restriction functor $\Fct(\mathrm{PI} \Cc,\presh(\Dd)) \to \Fct(\mathrm I\Cc,\presh(\Dd))$ admits a right adjoint, given by the right Kan extension. Let us fix their notation
\[
\mymatrix{
\Fct(\mathrm{T} \Cc,\presh(\Dd)) \ar@<2pt>[r]^\delta & \Fct(\mathrm{PI} \Cc,\presh(\Dd)) \ar@<2pt>[r]^\beta \ar@<2pt>[l]^\gamma & \Fct(\mathrm{I} \Cc,\presh(\Dd)) \ar@<2pt>[l]^\alpha
}
\]
the left adjoints being represented above their right adjoint. Note that both $\alpha$ and $\delta$ are fully faithful.
Let also $\tau$ denote the fully faithful functor
\[
\mymatrix{
\displaystyle \Fct_m(\Cc,\Dd) \simeq \Fct^{\mathrm c}(\mathrm I \Cc, \Dd) \timesunder[\Fct(\Cc,\Dd)] \Fct^{\mathrm l}(\mathrm P \Cc,\Dd) \ar[r]^-\tau &
\displaystyle \Fct(\mathrm I \Cc, \presh(\Dd)) \timesunder[\Fct(\Cc,\presh(\Dd))] \Fct^{\mathrm l}(\mathrm P \Cc,\presh(\Dd)) \simeq \Fct(\mathrm I \Cc, \presh(\Dd))
}
\]
where $\Fct^\mathrm c$ (resp. $\Fct^\mathrm l$) denotes the category of functors preserving filtered colimits (resp. cofiltered limits) which exist in the source. We use here that the Yoneda embedding $\Dd \to \presh(\Dd)$ preserves limits.
Let $\theta$ be the fully faithful functor
\[
\mymatrix{
\Fct_t(\mathrm T\Cc,\Dd) \ar[r]^-\theta &
\Fct(\mathrm T\Cc,\presh(\Dd))
}
\]
The composite functor $\beta \delta$ is nothing but the restriction along the canonical inclusion $\mathrm I \Cc \to \mathrm T \Cc$.
It follows that $\beta \delta \theta$ has image in the essential image of $\tau$. On the other hand, the functor $\gamma \alpha \tau$ has image in the essential image of $\theta$.
We hence get an adjunction
\[
f \colon \Fct_t(\mathrm T\Cc,\Dd) \rightleftarrows \Fct_m(\Cc,\Dd) \noloc g
\]
where $f$ is left adjoint to $g$. The functor $g$ is equivalent to the restriction functor and the unit transformation $fg \to \id_X$ is then an equivalence. Moreover, as objects of $\mathrm T\Cc$ are either pro-objects or ind-objects, the restriction functor $f$ is conservative. It follows that the above adjunction is an equivalence.
\end{proof}

\begin{cor}\label{cortateinindpro}
The category of Tate objects is equivalent to the smallest stable and idempotent complete full subcategory of $\Indu U \Prou U(\Cc)$ generated by the images of $\Indu U(\Cc)$ and $\Prou U(\Cc)$.
\end{cor}
\begin{proof}
This follows from \autoref{univpropT0} and \autoref{tatification}.
\end{proof}

\begin{rmq}
The fully faithful functor $\operatorname{j} \colon \Tateu U_0(\Cc) \to \Prou U\Indu U(\Cc)$ preserves both the limits and colimits which exist in $\Tateu U_0(\Cc)$. Let indeed $\bar x \colon K \to \Tateu U_0(\Cc)$ be a diagram which admits a colimit $x \in \Tateu U_0(\Cc)$. Let us denote by $x'$ a colimit of $\operatorname{j} \bar x$ in $\Prou U\Indu U(\Cc)$. We have, for any cofiltered diagram $y \bar \colon L\op \to \Indu U(\Cc)$
\begin{align*}
\Map_{\Prou U\Indu U(\Cc)}(x', \lim \bar y) &\simeq \colim_l \colim_k \Map_{\Prou U\Indu U(\Cc)}(\operatorname{j} \bar x, \bar y) \simeq \colim_l \colim_k \Map_{\Tateu U_0(\Cc)}(\bar x, \bar y)
\\ &\simeq \colim_l \Map_{\Tateu U_0(\Cc)}(x,\bar y) \simeq \colim_l \Map_{\Prou U\Indu U(\Cc)}(x,\bar y)
\\ &\simeq \colim_l \Map_{\Prou U\Indu U(\Cc)}(x,\bar y) \simeq \Map_{\Prou U\Indu U(\Cc)}(x,\lim \bar y)
\end{align*}
We show symmetrically that the inclusion $\Tateu U_0(\Cc) \to \Indu U \Prou U(\Cc)$ preserves limits.
It follows that limits and colimits that exist in $\Tateu U_0(\Cc)$ are exactly those coming from diagram in either $\Indu U(\Cc)$ or $\Prou U(\Dd)$. We can hence reformulate the universal property from \autoref{univpropT0} as follows:
The datum of a commutative square
\[
\mymatrix{
\Cc \ar[r] \ar[d] & \Indu U(\Cc) \ar[d]^f \\ \Prou U(\Cc) \ar[r]_-g & \Dd
}
\]
such that $f$ preserves filtered colimits and $g$ preserves cofiltered limits is equivalent to that of a functor $\Tateu U_0(\Cc) \to \Dd$ preserving both filtered colimits and cofiltered limits which exist in $\Tateu U_0(\Cc)$.
\end{rmq}

Let us close this section with the following lemma. 
\begin{lem}\label{tate-dual}
Let $\Cc$ be a $\mathbb V$-small stable $(\infty,1)$-category with a functor $f \colon \Cc\op \to \Cc$.
The functor $f$ induces a functor
\[
\tilde f \colon \left(\Prou V\Indu U(\Cc)\right)\op \to \Prou V\Indu U(\Cc)
\]
which maps (elementary) $\mathbb U$-Tate objects to (elementary) $\mathbb U$-Tate objects.

If moreover the functor $f$ is an equivalence, then $\tilde f$ induces an equivalence 
\[
\left(\Tateu U(\Cc)\right)\op \simeq \Tateu U(\Cc)
\]
\end{lem}

\begin{rmq}
The above lemma applies for instance when $\Cc$ is the category of perfect complexes on a base $k$. The duality functor $\dual{(-)} = \Homint(-,k)$ induces a duality of Tate objects
\[
\dual{(-)} \colon \Tate_k = \Tate(\Perf_k) \to^\sim \Tate_k\op
\]
In particular, for any Tate object $X$, we have $\dual{(\dual X)} \simeq X$.
\end{rmq}

\begin{proof}
The category $\Prou V\Indu U(\Cc)$ has all $\mathbb V$-small limits and colimits --- it is the opposite category of a $\mathbb V$-presentable category. We define the functor $\tilde f$ as the extension of the composition
\[
\Cc\op \to \Cc \to \Prou V\Indu U(\Cc)
\]
It maps objects of $\Indu U(\Cc)$ to objects of $\Prou U(\Cc) \subset \Prou V(\Cc)$ and vice-versa and therefore preserves pure Tate objects. The functor $\tilde f$ also preserves finite limits. It follows that it preserves Tate objects.
\end{proof}

\section{Lattices}
\label{sectionlattices}
Tate object are characterised by the existence of a lattice. A lattice for a pro-ind-object $X$ is an exact sequence
\[
X^p \to X \to X^i
\]
where $X^p$ is a pro-object and $X^i$ is an ind-object. We will see below that a pro-ind-object is a Tate object if and only if it admits a lattice. We will then study the category of lattices of a given Tate object.

\begin{prop}\label{tatediagexist}
Let $\Cc$ be a $\mathbb V$-small stable $(\infty,1)$-category. For any elementary Tate objects $X \in \Tateu U_\mathrm{el}(\Cc)$ there exists a $\mathbb U$-small cofiltered diagram $\bar X \colon K\op \to \Indu U(\Cc)$ such that
\begin{itemize}
\item The object $X$ is a limit of $\bar X$ in $\Prou U \Indu U(\Cc)$ and
\item For any $k \in K$ the diagram $\ker\left( \bar X \to \bar X(k) \right) \colon \left(\comma{k}{K}\right)\op \to \Indu U(\Cc)$ has values in the essential image of $\Cc$.
\end{itemize}
\end{prop}

\begin{df}\label{tatediag}
Let $\Cc$ be a $\mathbb V$-small stable $(\infty,1)$-category. For any elementary Tate object $X \in \Tateu U_\mathrm{el}(\Cc)$, we will call a Tate diagram for $X$ any $\mathbb U$-small cofiltered diagram $\bar X \colon K\op \to \Indu U(\Cc)$ as in \autoref{tatediagexist}.
\end{df}

\begin{proof}[of \autoref{tatediagexist}]
Let $\Dd$ denote the full subcategory of $\Prou U \Indu U(\Cc)$ spanned by those objects $X$ satisfying the conclusion of the proposition.
The category $\Dd$ obviously contains both the essential images of $\Indu U(\Cc)$ and $\Prou U(\Cc)$.
It suffices to prove that $\Dd$ is stable by extension. We see that it is stable by shifts and we can thus consider an exact sequence $X \to X_0 \to X_1$ in $\Prou U \Indu U(\Cc)$ such that both $X_0$ and $X_1$ are in $\Dd$.
Let $\bar X_0 \colon K\op \to \Indu U(\Cc)$ and $\bar X_1 \colon L\op \to \Indu U(\Cc)$ be a $\mathbb U$-small cofiltered diagrams of whom $X_0$ and $X_1$ are limits in $\Prou U \Indu U(\Cc)$.
Using \autoref{strictification}, we can assume $K = L$ and that we have a diagram $K\op \to \Fct(\Delta^1, \Indu U(\Cc))$ of whom the map $X_0 \to X_1$ is a limit. Considering the pointwise kernel, we get a diagram $\bar X \colon K\op \to \Indu U(\Cc)$ of whom $X$ is a limit. It obviously satisfies the required property.
\end{proof}

\begin{rmq}\label{tatediagramsgivelattices}
To state the above proposition in an informal way, any elementary Tate object $X$ can be represented by a diagram $\lim_\alpha \colim_\beta X_{\alpha\beta}$ such that for any $\alpha_0$, the kernel of canonical projection $X \to \colim_\beta X_{\alpha_0\beta}$ is actually a pro-object. This obviously imply the following corollary.
\end{rmq}

\begin{cor}\label{latticesexist}
Any elementary Tate object $X$ fits into an exact sequence
\[
X^p \to X \to X^i
\]
where $X^p \in \Prou U(\Cc)$ and $X^i \in \Indu U(\Cc)$.
\end{cor}

\begin{df}
Let $\Cc$ be a $\mathbb V$-small stable $(\infty,1)$-category. For any elementary Tate object $X \in \Tateu U_\mathrm{el}(\Cc)$, we will call a lattice of $X$ any exact sequence
\[
X^p \to X \to X^i
\]
where $X^p \in \Prou U(\Cc)$ and $X^i \in \Indu U(\Cc)$.
\end{df}

\begin{rmq}\label{latticesgivetatediagrams}
Let $X$ be an elementary Tate object in $\Cc$. Let us consider a lattice $X^p \to X \to X^i$ of $X$. We will construct a Tate diagram for $X$ out of it. Let $\bar X^p \colon K\op \to \Cc$ be a $\mathbb U$-small cofiltered diagram of which $X^p$ is a limit in $\Prou U(\Cc)$.
The extension map $X^i[-1] \to X^p$ induces a natural transformation $X^i[-1] \to \bar X^p$ from the constant diagram $X^i[-1] \colon K\op \to \Indu U(\Cc)$ to $\bar X^p$.
The quotient of this natural transformation defines a diagram $\bar X \colon K\op \to \Indu U(\Cc)$ which is by construction a Tate diagram for $X$: for any morphism $k \to l$ in $K$, the kernel of the induced map $\bar X(l) \to \bar X(k)$ is equivalent to that of $\bar X^p(l) \to \bar X^p(k)$ which belongs to $\Cc$.
\end{rmq}

\begin{rmq}
In the literature, the word "lattice" is often dedicated to maps $X^p \to X$ whose quotient is an ind-object -- where $X^p$ is a pro-object.
\end{rmq}

\begin{lem}\label{compatiblelattice}
Let $\Cc$ be a $\mathbb V$-small stable $(\infty,1)$-category. Let $f \colon X \to Y$ be a map between elementary Tate objects in $\Cc$. For any lattice $Y^p \to Y \to Y^i$ there exists a lattice $X^p \to X \to X^i$ compatible with $f$, ie fitting in a commutative diagram
\[
\mymatrix{
X^p \ar[r] \ar[d] & X \ar[r] \ar[d]^f & X^i \ar[d] \\ Y^p \ar[r] & Y \ar[r] & Y^i
}
\]
Dually, for any lattice $X^p \to X \to X^i$ there exists a lattice $Y^p \to Y \to Y^i$ and a commutative diagram as above.
\end{lem}
\begin{rmq}
In particular, any map $X \to Y^i$ from a Tate object to an ind-object factors through a lattice $X \to X^i$ of $X$.
\end{rmq}
\begin{proof}
Let $\bar Y^p \colon K\op \to \Cc$ be a cofiltered diagram of whom $Y^p$ is a limit in $\Prou U(\Cc)$. The procedure of \autoref{latticesgivetatediagrams} defines a Tate diagram for $Y$
\[
\bar Y = \quot{\bar Y^p}{Y^i[-1]} \colon K\op \to \Indu U(\Cc)
\]
Strictifying the map $f$ with \autoref{strictification}, we get a cofinal map $\alpha \colon J \to K$ and a diagram $\theta \colon J\op \times \Delta^1 \to \Indu U(\Cc)$ such that $\theta(-,1) \simeq \bar Y(\alpha-)$.
The functor $\theta$ defines Tate diagrams $\theta(-,0)$ for $X$ and $\theta(-,1)$ for $Y$. Let us fix $j \in J$.
Using \autoref{tatediagramsgivelattices}, this defines lattices $X^p \to X \to X^i$ and $Y_0^p \to Y \to Y_0^i$, together with a commutative diagram
\[
\mymatrix{
X^p \ar[r] \ar[d] & X \ar[r] \ar[d]^f & X^i \ar[d] \\ Y_0^p \ar[r] & Y \ar[r] & Y_0^i
}
\]
Note that by definition we have $Y_0^i = \theta(j,1) = \quot{\bar Y^p(\alpha j)}{Y^i[-1]} \simeq \ker(Y^i \to \bar Y^p(\alpha j))$ and hence get the commutative diagram
\[
\mymatrix{
Y_0^p \ar[r] \ar[d] & Y \ar[r] \ar[d]^f & Y_0^i \ar[d] \\ Y^p \ar[r] & Y \ar[r] & Y^i
}
\]
To prove the dual statement, we use the equivalence $\Tateu U(\Cc\op) \simeq \left(\Tateu U(\Cc)\right)\op$.
\end{proof}

\newcommand{\lattices}{\mathbf{Latt}}
Let us now denote by $\lattices_\Cc$ the full-subcategory of $\Fct(\Delta^1 \times \Delta^1, \Prou U\Indu U(\Cc))$ spanned by the cocartesian (and hence also cartesian) squares of the form
\[
\mymatrix{
X^p \cart \ar[r] \ar[d] & X \ar[d] \\ 0 \ar[r] & X_i \cocart
}
\]
where $X^p$ lies in the essential image of $\Prou U(\Cc)$ and $X^i$ lies in that of $\Indu U(\Cc)$. Let us denote by $q$ the natural functor $\lattices_\Cc \to \Tateu U_\mathrm{el}(\Cc)$ mapping a square as above to $X$.
Let us also denote the $\pi^p$ and $\pi_i$ the natural functors from $\lattices_\Cc$ to $\Prou U(\Cc)$ and $\Indu U(\Cc)$ respectively.

\begin{df}
For any elementary Tate object $X$, we will denote by $\lattices_\Cc(X)$ the fibre category $q^{-1}(X)$. We will call it the category of lattices of $X$.
\end{df}

\begin{rmq}
In the literature, the category $\lattices_\Cc(X)$ defined above is sometimes called the Sato grassmannian of $X$.
\end{rmq}

Let us study morphisms between lattices. Note first that a lattice $X^p \to X \to X^i$ is determined by the morphism $X^p \to X$. A map between lattices is a commutative diagram
\[
\mymatrix{
X^p_0 \ar[r] \ar[d]_\alpha & X \ar[r] \ar[d]^= & X^i_0 \ar[d] \\ X^p \ar[r] & X \ar[r] & X^i
}
\]
It is actually determined by the quotient of $\alpha$, which belongs to $\Cc$. That is what the following lemma is about.
\begin{lem}\label{latticesdifferbyc}
Let $X^\bullet = (X^p \to X \to X^i)$ be a lattice. There is a canonical equivalence
\[
\quot{\lattices_\Cc(X)}{X^\bullet} \simeq \comma{X^p}{\Cc}
\]
\end{lem}

\begin{proof}
We consider the inclusion $\{(0,0)\} \to \Delta^1 \times \Delta^1$. It induces a functor $\lattices_\Cc \to \Prou U(\Cc)$. In particular, we get
\[
P \colon \lattices_\Cc(X)^{\Delta^1} \to \Prou U(\Cc)^{\Delta^1}
\]
Let $Q \colon \Prou U(\Cc)^{\Delta^1} \to \Prou U(\Cc)^{\Delta^1}$ denote the functor mapping a morphism $x \to y$ to the morphism $y \to \quot{y}{x}$.
Let us consider a morphism of lattices of $X$
\[
\mymatrix{
X_0^p \ar[d]_\alpha \ar[r] & X \ar[d]^= \ar[r] & X_0^i \ar[d]^\beta \\ X^p \ar[r] & X \ar[r] & X^i
}
\]
The quotient $x$ of $\alpha$ is equivalent to the shift of that of $\beta$. It follows from \autoref{indandproareinc} that $x$ belongs to the essential image of $\Cc$.
The composite functor
\[
\mymatrix{
\quot{\lattices_\Cc(X)}{X^\bullet} \ar[r]^-i & \lattices_\Cc(X)^{\Delta^1} \ar[r]^-P & \Prou U(\Cc)^{\Delta^1} \ar[r]^-Q_-{\sim} & \Prou U(\Cc)^{\Delta^1}
}
\]
has values in the essential image of $\comma{X^p}{\Cc}$.
This functor is moreover fully faithful. We get a fully faithful functor
\[
\phi \colon \quot{\lattices_\Cc}{X^\bullet} \to \comma{X^p}{\Cc}
\]
Let now $\gamma \colon X^p \to x$ be in $\comma{X^p}{\Cc}$. The quotient of the induced map $\ker(\gamma) \to X$ is an ind-object and thus defines a lattice of $X$. The functor $\phi$ is essentially surjective and hence an equivalence.
\end{proof}

We also let $\quot{\lattices_\Cc}{X}$ denote the category
\[
\mymatrix{
\quot{\lattices_\Cc}{X} \ar[rr] \ar[d] \cart[][10] && \{ X \} \ar[d] \\
\lattices_\Cc^{\Delta^1} \ar[r]_-{\mathrm{target}} & \lattices_\Cc \ar[r]_-q & \Tateu U_\mathrm{el}(\Cc)
}
\]
We have a natural fully faithful functor $f_X \colon \lattices_\Cc(X) \to \quot{\lattices_\Cc}{X}$. We define dually the category $\comma{X}{\lattices_\Cc}$ and the fully faithful functor $g_X \colon \lattices_\Cc(X) \to \comma{X}{\lattices_\Cc}$.

Let $i \colon \Prou U(\Cc) \to \Indu V\Prou U(\Cc)$ and $j \colon \Indu U(\Cc) \to \Prou V \Indu U(\Cc)$ denote the canonical embeddings.
Let us denote by $\pi_!^p$ the left Kan extension of $i \pi^p$ along the functor $q \colon \lattices_\Cc \to \Tateu U_\mathrm{el}(\Cc)$. Dually, we define $j \pi_i^!$ the right Kan extension of $\pi_i$ along $q$. We get a diagram
\[
\mymatrix{
\Prou U(\Cc) \ar[d]_i & \lattices_\Cc \ar[r]^-{\pi_i} \ar[l]_-{\pi^p} \ar[d]^q & \Indu U(\Cc) \ar[d]^j \\
\Indu V\Prou U(\Cc) & \Tateu U_\mathrm{el}(\Cc) \ar@{=>}[ul]^\alpha \ar@{<=}[ur]_\beta \ar[l]^-{\pi^p_!} \ar[r]_-{\pi_i^!} & \Prou V\Indu U(\Cc)
}
\]
\begin{lem}\label{kanextensions}
The functor $\pi^p_!$ is equivalent to the embedding $\Tateu U_\mathrm{el}(\Cc) \to \Indu V\Prou U(\Cc)$ defined in \autoref{cortateinindpro}.
The functor $\pi_i^!$ is equivalent to the canonical embedding $\Tateu U_\mathrm{el}(\Cc) \to \Prou V\Indu U(\Cc)$ defined in \autoref{dftate}.
\end{lem}

\begin{proof}
The statement about $\pi^!_i$ is dual to that about $\pi^p_!$. Let us prove the latter.
Let $s$ denote the section of $\pi^p$ mapping pro-object $X$ to the exact sequence $X \to X \to 0$ and let $t$ denote the section of $\pi_i$ mapping an ind-object $Y$ to $0 \to Y \to Y$.
It suffices to prove that the induced functors
\begin{align*}
&\Prou U(\Cc) \to^s \lattices_\Cc \to^q \Tateu U_\mathrm{el} \to^{\pi^p_!} \Indu V\Prou U(\Cc) \\
&\Indu U(\Cc) \to^t \lattices_\Cc \to^q \Tateu U_\mathrm{el} \to^{\pi^p_!} \Indu V\Prou U(\Cc) \\
\end{align*}
Let us first deal with the case of $\Prou U(\Cc)$. Let $X$ be a pro-object. The image $s(X)$ is a final object in the category $\quot{\lattices_\Cc}{X}$.
Hence the canonical map
\[
\alpha_{s(X)} \colon \pi^p_! q s(X) \simeq \colim_{Z^\bullet \in \quot{\lattices_\Cc}{X}} i \pi^p(Z^\bullet) \to i \pi^p s(X) \simeq i(X)
\]
is an equivalence.
Let now $Y$ be an ind-object.
Let us prove that the category $\quot{\Cc}{Y}$ of exact sequence $y \to Y \to \quot{Y}{y}$ where $y \in \Cc$ is cofinal in $\quot{\lattices_\Cc}{Y}$. This obviously implies the result.
To prove this cofinality, we will use Quillen's theorem A. Let us denote by $g$ the functor $\quot{\Cc}{Y} \to \quot{\lattices_\Cc}{Y}$.
Let $Z^\bullet = Z^p \to Z \to Z^i$ be a lattice with a map $Z \to Y$.
From Quillen's theorem A (see \cite[4.1.3.1]{lurie:htt}),  it suffices to prove that the simplicial set
\[
K = \comma{Z^\bullet}{\left(\quot{\Cc}{Y}\right)} = \quot{\Cc}{Y} \times_{\quot{\lattices_\Cc}{Y}} \comma{Z^\bullet}{\left(\quot{\lattices_\Cc}{Y}\right)}
\]
is contractible. We have a diagram
\[
\mymatrix{
Z^p \ar[r] & Z \ar[d] \ar[r] & Z^i \\ & Y &
}
\]
The composite map $Z^p \to Y$, from a pro-object to an ind-object, factors through an object $y \in \Cc$. Hence we get a commutative diagram
\[
\mymatrix{
Z^p \ar[r] \ar[d] & Z \ar[d] \ar[r] & Z^i \ar[d] \\ y \ar[r] & Y \ar[r] & \quot{Y}{y}
}
\]
This proves that $K$ is not empty. The category $\Cc$ admits finite colimits and it follows that $K$ is filtered. The result is then deduced from \cite[5.5.8.7]{lurie:htt}.
\end{proof}

\begin{rmq}
Let $X$ be a Tate object in $\Cc$ and let $X_0^\bullet = (X_0^p \to X \to X_0^i)$ and $X^\bullet_1 = (X_1^p \to X \to X^i_1$ be two lattices for $X$. There is a lattice $X^\bullet$ for $X$ with maps $X_0^\bullet \from X^\bullet \to X_1^\bullet$.
To prove this statement, let us use \autoref{latticesgivetatediagrams}. It defines two Tate diagrams
\begin{align*}
&\bar X_0 \colon K\op \to \Indu U(\Cc) \\
&\bar X_1 \colon L\op \to \Indu U(\Cc)
\end{align*}
for $X$. Strictifying the identity of $X$ using \autoref{strictification}, we get a diagram
\[
\theta \colon J\op \times \Delta^1 \to \Indu U(\Cc)
\]
with cofinal maps $\alpha \colon J \to K$ and $\beta \colon J \to L$ such that $\theta(-,0) \simeq \bar X_0(\alpha -)$ and $\theta(-,1) \simeq \bar X_1(\beta -)$.
The diagram $\theta(-,1)$ is again a Tate diagram for $X$ and hence defines a lattice $X^\bullet$ for $X$. It naturally comes with morphisms $X_0^\bullet \from X^\bullet \to X_1^\bullet$.
\end{rmq}

We will improve the above remark into the following $\infty$-categorical incarnation of a phenomenon first discovered in \cite[theorem 6.7]{bgw:tate}.

\begin{thm}\label{latticesfiltered}
Let $\Cc$ be a $\mathbb U$-small stable and idempotent complete $(\infty,1)$-category.
For any elementary Tate object $X$ in $\Cc$, the category $\lattices_\Cc(X)$ is $\mathbb U$-small and both filtered and cofiltered.
Moreover the functor $f_X \colon \lattices_\Cc(X) \to \quot{\lattices_\Cc}{X}$ is cofinal and the functor $g_X \colon \lattices_\Cc(X) \to \comma{X}{\lattices_\Cc}$ is coinitial.
\end{thm}

This theorem, together with \autoref{kanextensions}, implies the following
\begin{cor}\label{colimoflattices}
Any elementary Tate object $X$ is the colimit in $\Indu U \Prou U(\Cc)$
\[
X \simeq \colim X^p
\]
and the limit in $\Prou U\Indu U(\Cc)$
\[
X \simeq \lim X^i
\]
where the limit and the colimit are indexed by $(X^p \to X \to X^i) \in \lattices_\Cc(X)$.
\end{cor}

\begin{lem}\label{filteredsmall}
Let $K$ be a $\mathbb V$-small simplicial set. Assume that $K$ is filtered and that for any vertex $k \in K$, the simplicial sets $\comma{k}{K}$ and $\quot{K}{k}$ are $\mathbb U$-small. Then the simplicial set $K$ is $\mathbb U$-small.
\end{lem}
\begin{proof}
Let $k \in K$ be any vertex. As $K$ is filtered we have
\[
K = \bigcup_{l \in \comma{k}{K}} \quot{K}{l}
\]
The result follows.
\end{proof}

\begin{proof}[of \autoref{latticesfiltered}]
Let us delay the size issue. We will prove that $\lattices_\Cc(X)$ is filtered and that $f_X$ is cofinal. What remains is deduced using the equivalence $\Tateu U(\Cc\op) \simeq (\Tateu U(\Cc))\op$.

We first say that $\lattices_\Cc(X)$ is not empty -- see \autoref{tatediagramsgivelattices}. Let now $\bar X^\bullet \colon K \to \lattices_\Cc(X)$ be a finite diagram.
We consider the composite diagram
\[
f_X \bar X^\bullet \colon K \to \quot{\lattices_\Cc}{X}
\]
The category $\quot{\lattices_\Cc}{X}$ admits finite colimits, and we can hence extend $f_X \bar X^\bullet$ into a colimit diagram $K^\triangleright \to \quot{\lattices_\Cc}{X}$.
Let us denote by $Y^\bullet = (Y^p \to Y \to Y^i)$ the colimit. It comes with a map $\psi \colon Y \to X$.
We get a lattice $X^\bullet = (X^p \to X \to X^i)$ from \autoref{compatiblelattice}, with a map $X^\bullet \to Y^\bullet$ lifting $\psi$.
Using the composition in $\quot{\lattices_\Cc}{X}$, we get a diagram  $K^\triangleright \to \quot{\lattices_\Cc}{X}$ whose vertices lie in the essential image of $f_X$. Since $f_X$ is fully faithful, we get a diagram
\[
K^\triangleright \to \lattices_\Cc(X)
\]
extending $\bar X^\bullet$. This proves the category $\lattices_\Cc(X)$ is filtered.

We now have to prove the functor $f_X$ is cofinal. Let $Y^\bullet \in \quot{\lattices_\Cc}{X}$. From \autoref{compatiblelattice}, we deduce that the category
\[
\lattices_\Cc(X) \times_{\quot{\lattices_\Cc}{X}} \comma{Y^\bullet}{\left(\quot{\lattices_\Cc}{X}\right)}
\]
is not empty. It is moreover filtered and hence the underlying simplicial set is contractible. We conclude using Quillen's theorem A -- see \cite[4.1.3.1]{lurie:htt}.

To see that $\lattices_\Cc(X)$ is essentially $\mathbb U$-small, we now use \autoref{filteredsmall} and \autoref{latticesdifferbyc}.
\end{proof}

\section{K-theory}

In this section, we will prove \autoref{introksusp}. The strategy of the proof is inspired by that in the case of exact categories, which can be found in \cite{saito:deloop}.

\begin{df}
Let $\Cc \to \Dd$ be a fully faithful exact functor between $\mathbb V$-small stable and idempotent complete $(\infty,1)$-categories. We denote by $\quot{\Dd}{\Cc}$ the cofibre of the functor $\Cc \to \Dd$ in the category of ($\mathbb V$-small) stable and idempotent complete $(\infty,1)$-categories.
Note that its existence is guarantied by \cite[Part 5]{bgt:characterisationk}.
Let us call $\quot{\Dd}{\Cc}$ the Verdier quotient of $\Dd$ by $\Cc$.
\end{df}

\begin{prop}
Let $\Cc$ be a $\mathbb V$-small stable and idempotent complete $(\infty,1)$-category. The commutative diagram
\[
\mymatrix{
\Cc \ar[r] \ar[d] & \Indu U(\Cc) \ar[d]
\\
\Prou U(\Cc) \ar[r] & \Tateu U(\Cc)
}
\]
induces an equivalence between the Verdier quotients 
\[
\quot{\Indu U(\Cc)}{\Cc} \to^\sim \quot{\Tateu U(\Cc)}{\Prou U(\Cc)}
\]
\end{prop}

\begin{proof}
Let us fix the following notations
\[
\begin{array}{ll}
\mathrm I^\mathbb V \Cc = \Indu V(\Cc)  & \mathrm I^\mathbb V \mathrm I \Cc = \Indu V \Indu U(\Cc) \\
\mathrm I^\mathbb V \mathrm P \Cc = \Indu V \Prou U(\Cc) \hspace{5mm} & \mathrm I^\mathbb V \mathrm T \Cc = \Indu V \Tateu U(\Cc)
\end{array}
\]
We also set
\[
\mathcal E = \Indu V\left(\quot{\Indu U(\Cc)}{\Cc} \right) \text{~~~and~~~}
\mathcal E' = \Indu V\left(\quot{\Tateu U(\Cc)}{\Prou U(\Cc)} \right)
\]
The commutative diagram
\[
\mymatrix{\Cc \ar[r] \ar[d] & \Indu U(\Cc) \ar[d] \ar[r] & \quot{\Indu U(\Cc)}{\Cc} \ar[d]
\\
\Prou U(\Cc) \ar[r] & \Tateu U(\Cc) \ar[r] & \quot{\Tateu U(\Cc)}{\Prou U(\Cc)}
}
\]
induces the diagram of adjunctions between presentable stable $(\infty,1)$-categories
\[
\mymatrix{
\mathrm I^\mathbb V \Cc \ar@<2pt>[r] \ar@<-2pt>[d] &
\mathrm I^\mathbb V \mathrm I \Cc \ar@<2pt>[r]^-\beta \ar@<-2pt>[d]_f \ar@<2pt>[l]^\varepsilon &
\Ee \ar@<-2pt>[d]_p \ar@<2pt>[l]^-\alpha \\
\mathrm I^\mathbb V \mathrm P \Cc \ar@<2pt>[r] \ar@<-2pt>[u]_g &
\mathrm I^\mathbb V \mathrm T \Cc \ar@<2pt>[l]^e \ar@<-2pt>[u]_\phi \ar@<2pt>[r]^-b &
\Ee' \ar@<2pt>[l]^-a \ar@<-2pt>[u]_q
}
\]
We have represented here the left adjoints on top or on the left of their right adjoint.
It follows from \cite[5.12 and 5.13]{bgt:characterisationk} that the two lines in the above diagram are cofibre sequences of presentable stable $(\infty,1)$-categories.
Since $\quot{\Indu U(\Cc)}{\Cc}$ (resp. $\quot{\Tateu U(\Cc)}{\Prou U(\Cc)}$) is idempotent complete, it is equivalent to the category of compact objects in $\mathcal E$ (resp. $\mathcal E'$). It hence suffices to prove that $p$ and $q$ are equivalences.
We will prove the sufficient assertions
\begin{assertions}
\item The functor $p$ is fully faithful.\label{pff}
\item The functor $q$ is conservative.\label{qcons}
\end{assertions}
Let us start with \ref{pff}. Using \cite[5.5]{bgt:characterisationk}, we deduce that both $a$ and $\alpha$ are fully faithful. Moreover, the functor $f$ is also fully faithful, and it thus suffices to prove the equivalence $f \alpha \simeq ap$.
We have $b f \alpha \simeq p \beta \alpha \simeq p$.
It is now enough to prove that $f \alpha$ has values in the essential image of $a$ (so that $a b f \alpha \simeq f \alpha$).
To do so, we will show that for any object $x \in \mathrm I^\mathbb V \mathrm I \Cc$, if $\varepsilon(x)$ vanishes, then so does $ef(x)$.
Let $\bar x \colon K \to \Indu U(\Cc)$ denote a $\mathbb V$-small filtered diagram whose colimit in $\mathrm I^\mathbb V \mathrm I \Cc$ is $x$. Let also $\bar y \colon L\op \to \Cc$ be a $\mathbb U$-small cofiltered diagram. We denote by $y$ its limit in $\Prou U(\Cc)$.
The image $ef(x)$ is the functor $\Prou U(\Cc) \to \sSets$ mapping $y$ to the simplicial set
\[
\colim_{k \in K} \colim_{l \in L} \Map_{\Indu U(\Cc)}(\bar y(l),\bar x(k)) \simeq
\colim_{l \in L} \colim_{k \in K} \Map_{\Indu U(\Cc)}(\bar y(l),\bar x(k))
\]
On the other hand, the assumption $\varepsilon(x) = 0$ implies that for any $c \in \Cc$, the space
\[
\colim_{k \in K} \Map_{\Indu U(\Cc)}(c,\bar x)
\]
is contractible. It follows from \cite[5.5.8.7]{lurie:htt} that $ef(x)$ vanishes.

We can now focus on \ref{qcons}. Since $q$ preserves exact sequences and $a$ is fully faithful, it suffices to prove that if $z \in \mathrm I^\mathbb V \mathrm T \Cc$ is such that both $\phi(z)$ and $e(z)$ vanish, then so does $z$.
We can see $z$ as a functor $\Tateu U(\Cc)\op \to \sSets$ preserving finite limits while $\phi(z)$ and $e(z)$ are its restriction respectively to $\Indu U(\Cc)\op$ and $\Prou U(\Cc)$. As $\Tateu U(\Cc)$ is generated by ind- and pro-objects under finite limits and retracts, we deduce that $z$ is equivalent to $0$.
\end{proof}

\begin{cor}\label{ktheorysusp}
Let $\Cc$ be a $\mathbb V$-small stable and idempotent complete $(\infty,1)$-category.
The spectrum of non-connective K-theory of $\Tateu U(\Cc)$ is the suspension of the non-connective $K$-theory of $\Cc$:
\[
\mathbb K(\Tateu U(\Cc)) \simeq \Sigma \mathbb K(\Cc)
\]
\end{cor}
\begin{rmq}
This corollary is an $\infty$-categorical version of a theorem of Sho Saito in exact $1$-categories in \cite{saito:deloop}.
\end{rmq}
\begin{proof}
Let us use the notations $\mathrm{I} \Cc = \Indu U(\Cc)$, $\mathrm P \Cc = \Prou U(\Cc)$ and $\mathrm T \Cc = \Tateu U(\Cc)$. 
Because the K-theory functor preserves cofibre sequences of stable categories (see \cite[sect. 9]{bgt:characterisationk}), we get two exact sequences in the $(\infty,1)$-category of spectra
\begin{align*}
\mathbb K (\Cc) \to \mathbb K (&\mathrm I \Cc) \to \mathbb K \left(\quot{\Indu U(\Cc)}{\Cc}\right) \\
\mathbb K (\mathrm P \Cc) \to \mathbb K (&\mathrm T \Cc) \to \mathbb K \left(\quot{\Indu U(\Cc)}{\Cc}\right)
\end{align*}
The vanishing of $\mathbb K(\mathrm P \Cc)$ and $\mathbb K (\mathrm I \Cc)$ -- since those categories contain countable sums -- concludes the proof.
\end{proof}

\section{An application: families of Tate complexes}
\label{sectionstacks}

In this last section, we will study Tate complexes in (derived) algebraic geometry. In this context, one should think of Tate complexes as \emph{structured} infinite dimensional vector bundles (or more generally quasicoherent complexes).
In this section, we will produce additive invariants on such Tate complexes, out of the common additive invariants of finite dimensional vector bundles (or perfect complexes), using our \autoref{ktheorysusp}.

As an example of such additive invariants, we will be able to define the dimension of a Tate complex (or rather its Euler characteristic). Given a finite dimensional vector bundle on a variety $X$, the dimension can be seen as a locally constant function $X \to \Z$ -- or equivalently as a class in $\homol^0(X,\Z)$.
The shift in K-theory we proved in \autoref{ktheorysusp} we allow us to define the dimension of a Tate complex as a class in $\homol^1(X,\Z)$ -- or equivalently as a $\Z$-torsor over $X$. 

Another very interesting example of such an invariant will be the determinant of a Tate complex. This determinant will be a class in $\homol^2(X,\Gm)$, hence classifying a gerbe with lien $\Gm$ over $X$.

Note that this question was, at least partially, addressed in \cite{drinfeld:tate} or \cite[section 3.2]{osipovzhu:categorical} in the context of (usual) algebraic geometry.

In the work, we will focus on the derived algebro-geometric setting, that has not been covered in any previous work.

We start with a short introduction to derived algebraic geometry.

\paragraph*{DAG in a nutshell:}
Let us assume $k$ is a field.
First introduced by Toën and Vezzosi in \cite{toen:hagii}, derived algebraic geometry is a generalisation of algebraic geometry in which we replace commutative algebras over $k$ by simplicial commutative algebras up to homotopy.
We refer to \cite{toen:dagems} for a recent survey of this theory.

Derived algebraic geometry allows us to study ill-behaved geometric situations. The most emblematic example is the study of non-generic intersections, or of quotients by a wild action.
Note that usual objects of algebraic geometry -- varieties, schemes, algebraic spaces of stacks -- embed in derived algebraic geometry.
Another nice feature of this theory is the cotangent. If we usually require smoothness to define a tangent bundle, dual to the cotangent, it is no longer needed in derived algebraic geometry (we only need finiteness conditions).
The main trick is to consider the (co)tangent not as a quasi-coherent sheaf, but as a complex of such.
The category of quasi-coherent complexes becomes a central object in this context. This category is actually a stable and idempotent complete $(\infty,1)$-category. This core example of such a category motivates the results of this article.
The derived category of quasi-coherent complex of a derived stack $X$ admits a full-subcategory $\Perf(X)$ of so-called perfect complexes. Perfect complexes are to complexes what finitely generated projective modules are to modules.
In particular, they behave regarding duality.

We will denote by $\sCAlg_k$ the $(\infty,1)$-category of simplicial commutative algebras over $k$. It is the $(\infty,1)$-localization of a model category along weak equivalences. 
Let us denote $\dAff_k$ the opposite $(\infty,1)$-category of $\sCAlg_k$. It is the category of derived affine schemes over $k$.

A derived prestack is a presheaf $\dAff_k\op \simeq \sCAlg_k \to \sSets$. We will thus write $\presh(\dAff_k)$ for the $(\infty,1)$-category of derived prestacks.
A derived stack is a prestack satisfying the étale descent condition. We will denote by $\dSt_k$ the $(\infty,1)$-category of derived stacks.
It comes with an adjunction
\[
(-)^+ \colon \presh(\dAff_k) \rightleftarrows \dSt_k
\]
where the left adjoint $(-)^+$ is called the stackification functor.

Let $\Perf$ denote the derived stack of perfect complexes $A \mapsto \Perf(A)$.
Let also $\mathrm{K}$ denote the connective K-theory functor (seen as a group object in spaces).
\begin{df}
Let us define the groups in prestacks
\begin{align*}
&\mathrm{K}^{\Perf} \colon A \mapsto \mathrm{K}(\Perf(A))\\
&\mathrm K^{\Tate} \colon A \mapsto \mathrm K(\Tateu U(\Perf(A)))
\end{align*}
We also define the prestack of Tate complexes
\[
\Tate \colon A \mapsto \Tateu U(\Perf(A))
\]
\end{df}

From \autoref{ktheorysusp} we get an exact sequence
\[
\mymatrix{
\B \mathrm K^{\Perf} \ar[r] & \mathrm K^{\Tate} \ar[r] & \mathrm K_0^{\Tate}
}
\]
\begin{lem}\label{k0vanishes}
The prestack $\mathrm K_0^{\Tate}$ vanishes Nisnevich-locally.
It follows that the map $\B \mathrm K^{\Perf} \to \mathrm K^{\Tate}$ is a Nisnevich-local equivalence.
\end{lem}

\begin{proof}
Is suffices to prove that for any Henselian simplicial commutative algebra $A$, we have
\[
\mathrm K_{-1}(\Perf(A) \simeq \mathrm K_0(\Tateu U(\Perf(A))) \simeq 0
\]
Recall that $A$ is Henselian if and only if $\pi_0(A)$ is. Using the Bass exact sequences, we get
\[
\mymatrix{
\mathrm K_0(A[t]) \oplus K_0(A[t^{-1}]) \ar[r] \ar[d]^f & \mathrm K_0(A[t,t^{-1}]) \ar[r] \ar[d]^g & \mathrm K_{-1}(A) \ar[r] \ar[d]^h & 0 \\
\mathrm K_0(\pi_0(A)[t]) \oplus K_0(\homol^0(A)[t^{-1}]) \ar[r] & \mathrm K_0(\pi_0(A)[t,t^{-1}]) \ar[r] & \mathrm K_{-1}(\pi_0(A)) \ar[r] & 0
}
\]
Since $\mathrm K_0$ only depends on the non-derived part of an affine scheme (see \cite[2.3.2]{waldhausen:ktheory}), both $f$ and $g$ are isomorphisms and hence so is $h$. We can thus restrict to the non-derived case -- which can be found in \cite[theorem 3.7]{drinfeld:tate}.
\end{proof}

\begin{thm}\label{detclass}
Let $i$ be any additive invariant of perfect complexes: it can then be encoded as a group morphism $i \colon \mathrm K^{\Perf} \to G$ for any group object $G$. 

The invariant $i$ induces an additive invariant $[i]$ of Tate complexes:
\[
[i] \colon \mathrm K^{\Tate} \to \B G
\]
with values in the classifying stack $\B G$.
In particular, for any derived algebraic stack $X$ and any Tate complex $E$ over $X$, we get a $G$-bundle classified by the map
\[
\mymatrix{
X \ar[r]^-E & \Tate \ar[r] & \mathrm K^{\Tate} \ar[r]^-{[i]} & \B G
}
\]
\end{thm}
\begin{rmq}
In the theorem above, we only need the map $i$ to preserve the group structure. It does not need to be preserve the commutativity constraints.
\end{rmq}

\begin{proof}
From \autoref{k0vanishes}, we see that the stack $(\mathrm K^{\Tate})^+$ associated to the prestack $\mathrm K^{\Tate}$ is equivalent to that associated to $\B \mathrm K^{\Perf}$.
We can hence form $\B i \colon \B \mathrm K^{\Perf} \to \B G$.
Let us denote by $\delta \colon (\B \mathrm K^{\Perf})^+ \to \B G$ the map of stacks obtained by stackifying $\B i$.
We can hence set
\[
\mymatrix{
[i] \colon \mathrm K^{\Tate} \ar[r] & (K^{\Tate})^+ \simeq (\B \mathrm K^{\Perf})^+ \ar[r]^-{\delta} & \B G
}
\]
\end{proof}

\paragraph{A motivating example:}
We define the determinantal anomaly to be the invariant $[\Det]$ associated to the determinant $\Det \colon \mathrm K^{\Perf} \to \B \Gm$:
\[
[\Det] \colon \mathrm K^{\Tate} \to \B \B \Gm \simeq \mathrm K(\Gm,2)
\]
In particular, any Tate object $E$ over a derived stack $X$ defines a determinantal anomaly $[\Det_E] \in \homol^2(X, \Oo_X^{\times})$.

The above construction is for instance useful in the following application. In \cite{hennion:floops}, the author introduces the $d$-dimensional formal loops space $\mathcal L ^d(X)$ with values in a nice enough derived Artin stack.
It is a derived stack representing maps from the punctured formal neighbourhood $\widehat \A^d \smallsetminus \{0\}$ to $X$.
We prove in \loccit that the tangent of this formal loops space is a Tate object over $\mathcal L^d(X)$. It follows from the above construction the existence of a class, called the determinantal class
\[
[\mathrm{det}_\T] \in \mathrm H^2(\mathcal L^d(X), \Oo^{\times})
\]
This class generalises the class introduced by Kapranov and Vasserot in \cite{kapranovvasserot:loop1}, that is proved to be an obstruction to the existence of sheaves of "chiral differential operators".

For instance, when $X$ is the stack $\B G$ classifying $G$-bundles, for an algebraic group $G$, this determinantal class determines a central extension of the tangent dg-Lie algebra of $\mathcal L^d(X)$ at the neutral element. This dg-Lie algebra is a first step toward higher dimensional Kac-Moody algebras.
Indeed, when $d=1$, we find back $\mathfrak g \otimes k(\!(t)\!)$ and its usual extension by a central charge. This direction is currently being studied in a joint work with Giovanni Faonte and Mikhail Kapranov.

\clearpage
\phantomsection
\addcontentsline{toc}{section}{\iflanguage{francais}{Références}{References}}

\end{document}